\documentclass[11pt,dvips,twoside,letterpaper]{article}
\usepackage[usenames]{color}
\usepackage[utf8]{inputenc}
\usepackage[T1]{fontenc} 
\usepackage{fancyhdr}
\usepackage{graphicx}
\usepackage{geometry}

\RequirePackage{amsfonts,amssymb,amsmath,amscd,amsthm}
\RequirePackage{txfonts}
\RequirePackage{graphicx}
\RequirePackage{xcolor}

\def\figurename{Figure} 
\makeatletter
\renewcommand{\fnum@figure}[1]{\figurename~\thefigure.}
\makeatother

\def\tablename{Table} 
\makeatletter
\renewcommand{\fnum@table}[1]{\tablename~\thetable.}
\makeatother

\usepackage[french,english]{babel}
\usepackage{amsmath}
\usepackage{amssymb}
\usepackage{amsfonts}
\usepackage{amsthm,amscd}

\newtheorem{theorem}{Theorem}[section]

\newtheorem{proposition}[theorem]{Proposition}
\theoremstyle{definition}

\theoremstyle{remark}

\numberwithin{equation}{section}

\usepackage{mathtools}
\usepackage{empheq}
\begin{document}

\title{Existence and multiplicity of solutions for a class of quasilinear problems in Orlicz-Sobolev spaces}\bigskip

\author{\bfseries\scshape Karima Ait-Mahiout\thanks{e-mail: karima\_ait@hotmail.fr} \, and \, Claudianor O. Alves\thanks{C.O. Alves was partially supported by CNPq/Brazil  304036/2013-7  and INCT-MAT, e-mail:  coalves@mat.ufcg.edu.br} }

\date{}
\maketitle
\thispagestyle{empty} \setcounter{page}{1}

\begin{abstract} \noindent
This work is concerned with the existence and multiplicity of solutions for the following class of quasilinear problems
$$
-\Delta_{\Phi}u+\phi(|u|)u=f(u)~\text{in} ~\Omega_{\lambda}, u(x)>0 ~\text{in}~\Omega_{\lambda}, u=0~ \mbox{on} ~\partial\Omega_{\lambda},
$$
where $\Phi(t)=\int_0^{|t|} \phi(s) s \, ds $ is an $N-$function, $\Delta_{\Phi}$ is the $\Phi-$Laplacian operator, \linebreak $\Omega_{\lambda}=\lambda \Omega,$ $\Omega$ is a smooth bounded domain in $\mathbb{R}^N,$ $N \geq 2$, $\lambda$ is a positive parameter and $f: \mathbb{R}\rightarrow \mathbb{R}$ is a continuous function. Here, we use variational methods to get multiplicity of solutions by using of Lusternik-Schnirelmann category of ${\Omega}$ in itself.
\end{abstract}

\vspace{.08in} \noindent \textbf{Keywords}: Variational methods, Quasilinear problems, Orlicz-Sobolev space, Positive solutions.

\allowdisplaybreaks

\section{Introduction}

In this paper, we study the existence of multiple solutions for the following class of quasilinear problems 
\begin{equation}\tag{$P_\lambda$}\label{P}
\begin{cases}
 -\Delta_{\Phi}u+\phi(|u|) u= f(u),~ ~\text{in}~ \Omega_{\lambda}\\ 
u>0, \quad \text{in}~ \Omega_{\lambda},\\ 
u=0, \quad \mbox{on} \quad \partial \Omega_\lambda,
\end{cases}
 \end{equation}
where $\Omega_\lambda = \lambda \Omega$, $\Omega \subset \mathbb{R}^N$ is a smooth bounded domain, $N\geq 2$, $\lambda$ is a positive parameter and $\Delta_{\Phi}u=div(\phi(|\nabla u|)\nabla u),$ where $\Phi(t)=\int_0^{|t|} \phi(s) ds,$ is the $\Phi-$Laplacian. We would like to detach that this type of operator appears in a lot of physical applications, such as:

\noindent {\it  Nonlinear Elasticity:} \, $ \Phi(t) = (1+t^{2})^{\alpha}-1, \alpha \in (1, \frac{N}{N-2})$, \\
\noindent { \it  Plasticity:} \, $ \Phi(t) = t^{p}\ln(1+t), 1< \frac{-1+\sqrt{1+4N}}{2}<p<N-1, N\geq 3$, \\
\noindent {\it  Non-Newtonian Fluid:  } \, $\Phi(t) = \frac{1}{p}|t|^{p}$ for $p>1$,\\
\noindent { \it  Plasma Physics: } \, $\Phi(t) = \frac{1}{p}|t|^{p} + \frac{1}{q}|t|^{q}$ where $1<p<q<N$ with $q \in (p, p^{*}).$

The reader can find more details about this subject in \cite{Db}, \cite{F1}, \cite{FN2} and their references.

In what follows, the function $\phi:[0,+\infty[\rightarrow [0,+\infty[$ is a $C^1-$ function which satisfies: 
   
\noindent $(\phi_1)$ $\phi(t)>0$ and 
$(\phi(t)t)'>0, ~\text{for all}~ t>0$.

\noindent $(\phi_2)$ There exist~$ l,m\in(1,N)$~ such that : 
$$ 
l\leq \frac{\phi(t)t^2}{\Phi(t)}\leq m, \quad \forall t\neq0, 
$$
\noindent where $l\leq m\leq l^{\star}=\frac{Nl}{N-l}.$ \\
\noindent $(\phi_3)~ \text{The function}~ \displaystyle\frac{\phi(t)}{t^{m-2}}~ \text{is nonincreasing in}~ (0,+\infty).$ 

\noindent $(\phi_4)$ The function $\phi$ is monotone.

\noindent $(\phi_5)$ There exists a constant $c>0$ such that  
$$
|\phi'(t)t|\leq c\phi(t), \quad \forall t\in [0,+\infty).
$$

We say that $\Phi\in \mathcal{C}_m$ if 
$$
\Phi(t)\geq |t|^m, ~\forall t\in \mathbb{R}.
$$

Moreover, we denote by $\gamma$ the following real number 
$$
\gamma=\begin{cases} m, ~~\text{if}~ \Phi\in \mathcal{C}_m,\\ l,~~~~\text{if}~ \Phi\notin \mathcal{C}_m.\end{cases}
$$

Here, we would like to mention that the functions $\phi$ associated with each N-function cited in this introduction, fulfill the conditions $(\phi_1)$-$(\phi_5)$.

Related to the function $f:\mathbb{R}\rightarrow \mathbb{R}$, we assume that it is a $C^1-$ function which satisfies the following conditions: \\ 
\mbox{} \quad $(f_1)$ There are functions $r,b:[0,+\infty)\rightarrow [0,+\infty)$ such that 
$$
\limsup_{|t|\rightarrow0}\frac{f'(t)}{(r(|t|)|t|)'}=0 \quad \text{and} \quad  \limsup_{|t|\rightarrow+\infty}\frac{|f'(t)|}{(b(|t|)|t|)'}<+\infty.
$$ 

There exists $\theta>m$ such that 

$(f_2)$\hspace{2cm}$0<\theta F(t)=\theta\int_{0}^t f(s) ds\leq t f(t), \quad \forall t\in \mathbb{R}\setminus\{0\}.$

$(f_3)$ The function $\displaystyle\frac{f(t)}{t^{m-1}}$ is strictly  increasing for $t>0.$

The functions $r$ and $b$ are $C^1-$functions which satisfy:

$(b_1)$ $b$ is increasing.

$(b_2)$ There exists a constant $\widehat{c}>0$ such that $$|b'(t)t|\leq \widehat{c} b(t), t\geq0.$$

$(b_3)$ There exist positive constants $b_1,b_2\in (1,\gamma^*)$ such that  
$$
b_1\leq \frac{b(t)t^2}{B(t)}\leq b_2~~\forall t\neq0, ~\text{where}~ B(t)=\int_0^{|t|} b(s) s ds ~\text{and} ~\gamma^*=\frac{N\gamma}{N-\gamma}.
$$
$(b_4)$ The function $B$ satisfies
$$
\limsup_{t\rightarrow 0}\frac{B(t)}{\Phi(t)}< +\infty \quad \mbox{and} \quad \limsup_{|t|\rightarrow +\infty}\frac{B(t)}{\Phi_{*}(t)}=0.
$$

\noindent $(r_1)$ $r$ is increasing.

\noindent $(r_2)$ There exists a constant $\overline{c}>0$ such that $$|r'(t)t|\leq \overline{c}r(t); \forall t\geq0.$$

\noindent $(r_3)$  There exist positive constants $r_1$ and $r_2$  such that $$r_1\leq \frac{r(t)t^2}{R(t)}\leq r_2, \forall t\neq 0, ~\text{where}~ R(t)=\int_0^{|t|} r(s) ds.$$

\noindent $(r_4)$ The function $R$ satisfies 
$$
\limsup_{t\rightarrow0}\frac{R(t)}{\Phi(t)}<+\infty~\text{and}~ \limsup_{|t|\rightarrow+\infty}\frac{R(t)}{\Phi_{*}(t)}=0,
$$
where $\Phi_{*}$ is the Sobolev conjugate function, which is defined by inverse function of
\[
G_{\Phi}(t) = \int_{0}^{t}\frac{\Phi^{-1}(s)}{s^{1+\frac{1}{N}}}ds.
\]

Hereafter, we use variational methods to get multiplicity of positive solutions for (\ref{P}), where the main idea is looking for critical points of the energy functional $I_\lambda: W_0^{1,\Phi}(\Omega_{\lambda})\rightarrow \mathbb{R}$ given by: 
$$
I_\lambda(u)=\int_{\Omega_{\lambda}} \Phi(|\nabla u|)\,dx+\int_{\Omega_{\lambda}} \Phi(|u|)\,dx-\int_{\Omega_{\lambda}} F(u)\,dx. 
$$
Using standard arguments, we know that $I_\lambda \in C^{1}(W_0^{1,\Phi}(\Omega_{\lambda}), \mathbb{R})$ with
$$
I_\lambda'(u)v=\int_{\Omega_{\lambda}} \phi(|\nabla u|)\nabla u \nabla v\, dx+\int_{\Omega_{\lambda}} \phi(|u|)uv\,dx-\int_{\Omega_{\lambda}} f(u)v\,dx, \quad \forall u,v \in W_0^{1,\Phi}(\Omega_{\lambda}).
$$ 
Hence, critical points of $I_\lambda$ are weak solutions of (\ref{P}). Hereafter, we denote by $c_{\lambda}$ the mountain pass level of  $I_\lambda$ and by $\mathcal{M}_\lambda$ the  set
$$
\mathcal{M}_\lambda=\{u \in W_0^{1,\Phi}(\Omega_\lambda) \setminus \{0\}\,:\,I'_{\lambda}(u)u=0\},
$$ 
which is the Nehari manifold associated with $I_\lambda$.

In the literature there are some works where the authors showed multiplicity of solutions for some related problems to (\ref{P}) by using of Lusternick- Schnirelman category of ${\Omega}$ in itself, denoted by $cat({\Omega})$, see for example for the case $\phi(t)=1$, Benci \& Cerami \cite{BC1, BC2, BC3}, Clap \& Ding \cite{CD}, Rey \cite{Rey} and Bahri \& Coron \cite{BaC}. For $\phi(t)=|t|^{p-2}$, with
$p\geq 2$, we cite the papers by Alves \cite{2005}, Alves \& Ding \cite{alvesD2} and references therein. Moreover, the reader can find in \cite{FN2}, \cite{FN3}, \cite{FN4}, \cite{MR1}, \cite{MR2}, \cite{MRep}, \cite{santos1}, \cite{RR} and \cite{TF}  recent results for some related problems to (\ref{P}) for $\lambda=1$. 

We would like to point out that if $Y$ is a closed subset of a topological space $X$, the Lusternik-Schnirelman category $cat_{X}(Y)$ is the least number of closed and contractible sets in $X$ which cover $Y$. If $X=Y$, we will use the notation $cat(X)$.

Motivated by results found in \cite{2005} and \cite{BC1}, in the present paper we have proved that the main results obtained in the mentioned papers also hold for a large class of ${\Phi}-$Laplacian operators, for example our main result can be used to prove multiple solutions for the following quasilinear problems:\\

\noindent {\bf Problem 1:}
$$
\left\{
	\begin{array}{l}
	- \mbox{div}\Big( \big(p|\nabla u|^{p-2}\ln(1+|\nabla u|) + \frac{|\nabla u|^{p-1}}{|\nabla u|+1} \big) \nabla u\Big) +
	\Big( p| u|^{p-2}\ln(1+| u|) + \frac{| u|^{p-1}}{| u|+1} \Big)u \, =\, f(u), \mbox{in} \quad \Omega_\lambda, \\
	u(x)>0, \quad \mbox{in} \quad \Omega_\lambda,
	\nonumber\\
	u=0, \quad \mbox{on} \quad \partial \Omega_\lambda, 
	\end{array}
	\right. \nonumber
$$
for $1< \frac{-1+\sqrt{1+4N}}{2}<p<N-1$ and $N\geq 3$. \\

\noindent {\bf Problem 2:}
$$
	\left\{
	\begin{array}{l}
	- 2\alpha\mbox{div}\big((1+|\nabla u|^2)^{\alpha -1}\nabla u\big) + 2\alpha (1+|u|^2)^{\alpha -1}u \, =\, f(u), \mbox{in} \quad \Omega_\lambda, \\
	u(x)>0, \quad \mbox{in} \quad \Omega_\lambda,
	\nonumber\\
	u=0, \quad \mbox{on} \quad \partial \Omega_\lambda,
	\end{array}
	\right.\nonumber 
$$
for $\alpha \in (1, \frac{N}{N-2})$ and $N >2$.

\vspace{0.5 cm}

Our main result is the following

\begin{theorem} \label{T1}
Assume $(\phi_1)-(\phi_5),(f_1)-(f_3), (r_1)-(r_4)$ and $(b_1)-(b_4)$. Then, there exists $\lambda^*> 0$ such that for $\lambda \geq \lambda^*$ problem (\ref{P}) has at least $cat(\Omega)$ of positive solutions. 	

\end{theorem}

The plan of the paper is as follows: In Section 2, we will fix some notations about Orlicz-Sobolev spaces and prove a compactness result for the energy functional associated with limit problem,  see Theorem \ref{31}. In Section 3, we study the behavior of some minimax levels and prove our main result. 

\vspace{0.5 cm}

\noindent \textbf{Notation:} In this paper, we use the following
notations:
\begin{itemize}
	\item  The usual norm in $W^{1,\Phi}(\mathbb{R}^N)$ will be denoted by
	$\|\;\;\;\|$.

	\item   $C$ denotes (possible different) any positive constant.
	
	\item   $B_{R}(z)$ denotes the open ball with center  $z$ and
	radius $R$ in $\mathbb{R}^N$.
	
	\item Since we are interested by finding positive solutions, we assume that 
$$
f(t)=0, ~\forall t\in(-\infty,0].
$$
	
\end{itemize}

\section{Preliminary results and notations}

In this section, we recall some properties of Orlicz-Sobolev spaces and show an important result of compactness for a special energy functional, which will be defined in Subsection 2.2. 

\subsection{Basics on Orlicz-Sobolev spaces}\label{expoents_variaveis} 

In this subsection, we recall some properties of Orlicz and Orlicz-Sobolev spaces. We refer to \cite{Adams, uni-conv, FN2, RR} for the fundamental properties of these spaces. First of all, we recall that a continuous function $\Phi : \mathbb{R} \rightarrow [0,+\infty)$ is a
N-function if:
\begin{description}
	\item[$(i)$] $\Phi$ is convex.
	\item[$(ii)$] $\Phi(t) = 0 \Leftrightarrow t = 0 $.
	\item[$(iii)$] $\displaystyle\lim_{t\rightarrow0}\frac{\Phi(t)}{t}=0$ and $\displaystyle\lim_{t\rightarrow+\infty}\frac{\Phi(t)}{t}= +\infty$ .
	\item[$(iv)$] $\Phi$ is even.
\end{description}
We say that a N-function $\Phi$ verifies the $\Delta_{2}$-condition, denote by  $\Phi \in \Delta_{2}$, if
\[
\Phi(2t) \leq K\Phi(t),\quad \forall t\geq 0,
\]
for some constant $K > 0$. In what follows, fixed an open set $A \subset \mathbb{R}^{N}$ and a N-function $\Phi$, we define the Orlicz space associated with $\Phi$ as
\[
L^{\Phi}(A) = \left\{  u \in L_{loc}^{1}(A) \colon \ \int_{A} \Phi\Big(\frac{|u|}{\lambda}\Big)dx < + \infty \ \ \mbox{for some}\ \ \lambda >0 \right\}.
\]
The space $L^{\Phi}(A)$ is a Banach space endowed with the Luxemburg norm given by
\[
\Vert u \Vert_{\Phi} = \inf\left\{  \lambda > 0 : \int_{A}\Phi\Big(\frac{|u|}{\lambda}\Big)dx \leq1\right\}.
\]
The complementary function $\widetilde{\Phi}$ associated with $\Phi$ is given
by the Legendre's transformation, that is,
\[
\widetilde{\Phi}(s) = \max_{t\geq 0}\{ st - \Phi(t)\}, \quad  \mbox{for} \quad s\geq0.
\]
The functions $\Phi$ and $\widetilde{\Phi}$ are complementary each other. In \cite{FN2, RR}, we find that
$$
\Phi, \widetilde{\Phi} \in \Delta_{2} \,\,\, \mbox{if, and only if,} \,\, (\phi_2) \,\,\, \mbox{holds}.
$$
Moreover, we also have a Young type inequality given by
\[
st \leq \Phi(t) + \widetilde{\Phi}(s), \quad \forall s, t\geq0.
\]
Using the above inequality, it is possible to prove a H\"older type inequality, that is,
\[
\Big| \int_{A}uvdx \Big| \leq 2 \Vert u \Vert_{\Phi}\Vert v \Vert_{\widetilde{\Phi}},\quad \forall u \in L^{\Phi}(A) \quad \mbox{and} \quad \forall v \in L^{\widetilde{\Phi}}(A).
\]
The corresponding Orlicz-Sobolev space is defined as
\[
W^{1, \Phi}(A) = \Big\{ u \in L^{\Phi}(A) \ :\ \frac{\partial u}{\partial x_{i}} \in L^{\Phi}(A), \quad i = 1, ..., N\Big\},
\]
endowed with the norm
\[
\Vert u \Vert = \Vert \nabla u \Vert_{\Phi} + \Vert u \Vert_{\Phi}.
\]

The space $W_0^{1,\Phi}(A)$ is defined as the closure of $C_0^{\infty}(A)$ with respect to Orlicz-Sobolev norm above.

The spaces $L^{\Phi}(A)$, $W^{1, \Phi}(A)$ and $W_0^{1, \Phi}(A)$ are separable and reflexive, when $\Phi$ and $\widetilde{\Phi}$ satisfy the $\Delta_{2}$-condition. The $\Delta_{2}$-condition implies that
\[
u_{n}\to u \ \mbox{in} \ \ L^{\Phi}(A)\quad \Leftrightarrow \quad \int_{A}\Phi(\vert u_{n} - u\vert)dx \rightarrow 0
\]
and
\[
u_{n}\to u \ \mbox{in} \ \ W^{1, \Phi}(A)\quad \Leftrightarrow  \quad \int_{A}\Phi(\vert \nabla u_{n} - \nabla u\vert)dx \to 0\,\, \mbox{and}\,\, \int_{A}\Phi(\vert u_{n} - u\vert)dx \to 0.
\]
In the literature, we have some important embeddings related to the Orlicz-Sobolev spaces. In \cite{Adams, DT}, it has been shown that if $B$ is a N-function with
\[
\limsup_{t\to 0}\frac{B(t)}{\Phi(t)}< +\infty \quad \mbox{and}\quad \limsup_{t\rightarrow +\infty}\frac{B(t)}{\Phi_{*}(t)}< +\infty,
\]
then the embedding
\[
W^{1, \Phi}(A)\hookrightarrow L^{B}(A)
\]
is continuous. If $A$ is a bounded domain, the embedding is compact.

\subsection{A compactness result for the limit problem}

From now on,  we denote by $I_{\infty}: W^{1,\Phi}(\mathbb{R}^N) \to \mathbb{R}$ the functional given by
$$
I_{\infty}(u)=\int_{\mathbb{R}^N} \Phi(|\nabla u|)\,dx+\int_{\mathbb{R}^N} \Phi(|u|)\,dx-\int_{\mathbb{R}^N} F(u)\,dx. 
$$
Using standard arguments, it is easy to prove that critical points of $I_\infty$ are weak solutions of the quasilinear problem
\begin{equation} \tag{$P_{\infty}$}
\begin{cases}
-\Delta_{\Phi}u+\phi(|u|)u= f(u),~\text{in}~\mathbb{R}^N,\\ u>0, ~\text{in}~ \mathbb{R}^N, \\ u\in W^{1,\Phi}(\mathbb{R}^N),
\end{cases}
\end{equation}
which is called {\it limit problem} associated with $(P_\lambda)$.

In \cite{JMP2014}, Alves and da Silva have proved that the above problem has a ground state solution $w \in W^{1,\Phi}(\mathbb{R}^N)$, that is, a solution which satisfies
$$
I_{\infty}(w)=c_\infty \quad \mbox{and} \quad I'_{\infty}(w)=0,
$$
where $c_\infty$ is the mountain pass level associated with $I_\infty$. Moreover, we also have 
$$
c_\infty=\inf_{u \in \mathcal{M}_\infty}I_\infty(u),
$$
where
$$
\mathcal{M}_\infty=\{u \in W^{1,\Phi}({\mathbb{R}^N}) \setminus \{0\}\,:\,I'_{\infty}(u)u=0\}.
$$ 
The set $\mathcal{M}_\infty$ is called the Nehari Manifold associated with $I_\infty$.

Next, we will prove an important result of compactness associated with functional $I_\infty$.

\begin{theorem} \label{310} (\textbf{Compactness theorem on Nehari manifold})~~ Let $(u_n)\subset W^{1,\Phi}(\mathbb{R}^N)$ be a sequence satisfy $$I_{\infty}(u_n)\rightarrow c_{\infty}~~\text{and}~~ u_n\in \mathcal{M}_{\infty}.$$
Then,\\
\noindent $i)$ \, $(u_n)$ is strongly convergent, \\
or \\
\noindent $ii)$ \, There exists $(y_n)\subset \mathbb{R}^N$ with $|y_n|\rightarrow\infty$ such that the sequence $v_n(x)=u_n(x+y_n)$ is strongly convergent to a function $v\in W^{1,\Phi}(\mathbb{R}^N)$ with  
$$
I_{\infty}(v)=c_{\infty}~~\text{and}~~ v\in \mathcal{M}_{\infty}.
$$
\end{theorem}
\begin{proof}~~ To begin with, we claim that $(u_n)$ is bounded in $W^{1,\Phi}(\mathbb{R}^N)$. Indeed, as $I_{\infty}(u_n)\rightarrow c_{\infty}$,  $(I_{\infty}(u_n))$ is bounded. Then, by  $(\phi_2),(f_2)$ and \protect{\cite[Lemma 2.3]{TMNA2014}}, 
$$
M\geq I_{\infty}(u_n)-\frac1\theta I'_{\infty}(u_n)u_n \geq 
\frac{(\theta-m)}{\theta}[\xi_0(||\nabla u_n||_{\Phi})+\xi_0(||u_n||_{\Phi})],
$$
for some positive constant $M$ and $\xi_{0}(t)=\min\{t^{l}, t^{m}\}$. Hence, $(u_n)$ is bounded in $W^{1,\Phi}(\mathbb{R}^N)$. Thereby, as $W^{1,\Phi}(\mathbb{R}^N)$ is a reflexive space there exists a subsequence 
of $(u_n)$, still denoted by $(u_n)$, and $u \in W^{1,\Phi}(\mathbb{R}^N)$ such that 
$$
u_n\rightharpoonup u~\text{in}~ W^{1,\Phi}(\mathbb{R}^N).
$$
By  Ekeland's Variational Principal, we can assume that $(u_n)$ satisfies  
$$
I'_{\infty}(u_n)-\gamma_n E'(u_n)=o_n(1),
$$
where $\gamma_n$ is a real number and $E_{\infty}(w)=I'_{\infty}(w)w, \forall w\in W^{1,\Phi}(\mathbb{R}^N).$

Using that $u_n\in \mathcal{M}_{\infty}$ together with $(\phi_3)$ and $(f_2)$, there exists $\delta>0$ such that: 
$$
|E'_{\infty}(w_n)w_n|\geq \delta, \forall n\in \mathbb{N}.
$$
Indeed, note that
\begin{align}\label{4.6}
-E'(u_n)u_n&=-\int_{\mathbb{R}^N}\phi'(|\nabla u_n|)|\nabla u_n|^3dx-\int_{\mathbb{R}^N}\phi'(|u_n|)|u_n|^3dx+\int_{\mathbb{R}^N} f'(u_n)u_n^2dx\nonumber\\
&\geq (1-m)\int_{\mathbb{R}^N}\phi(|\nabla u_n|)|\nabla u_n|^2dx+(1-m)\int_{\mathbb{R}^N}\phi(|u_n|)|u_n|^2 dx+\int_{\mathbb{R}^N}f'(u_n)u_n^2 dx\nonumber\\
&=(1-m)\int_{\mathbb{R}^N} f(u_n)u_n dx+\int_{\mathbb{R}^N} f'(u_n)u_n^2 dx\nonumber\\
&=\int_{\mathbb{R}^N}\left[f'(u_n)u_n^2-(m-1) f(u_n)u_n\right] dx.
\end{align}
Since $(u_n)$ is bounded and $||u_n||\kern.3em\not\kern -.3em \longrightarrow 0,$ by \cite[Theoreme 1.3]{TMNA2014} there is $(z_n)\subset \mathbb{R}^N$  such that $\widehat{u}_n(x)=u_n(x+z_n)$ is bounded in $W^{1,\Phi}(\mathbb{R}^N)$ and $\widehat{u}_n \rightharpoonup \widehat{u}$ in $W^{1,\Phi}(\mathbb{R}^N)$ with $\widehat{u}\neq0.$ Therefore,  there exists a measurable subset $A \subset \mathbb{R}^N$ with positive measure, such that $\widehat{u}>0$ a.e. in $A.$ Supposing  by contradiction that 
$$
\limsup_{n\rightarrow+\infty} E'(u_n)u_n=0,
$$
a simple change of variable in $(\ref{4.6})$, the condition $(f_4)$ and the Fatou's Lemma combine to give
$$
0\geq\int_{A} (f'(\widehat{u})\widehat{u}^2-(m-1)f(\widehat{u})\widehat{u}) dx>0,
$$ 
which is a contradiction. Thus, there exists $\delta>0$ such that
$$
|E'_{\infty}(u_n)u_n|>\delta, \forall n\in \mathbb{N}.
$$
As $I'_{\infty}(u_n)u_n =o_n(1)$,  we assure that $\gamma_n E'_{\infty}(u_n)u_n=o_n(1)$, which
yields $\gamma_n=o_n(1)$, and so, 
\begin{equation}\label{31}
I_{\infty}(u_n)\rightarrow c_{\infty}~\text{and}~I'_{\infty}(u_n)\rightarrow 0.
\end{equation}
  
Next, we will study the following situations:  $u\neq 0$ and $u=0$.
  
\noindent  \textbf{Case 1: $u\neq0$.} ~~~  From \protect{\cite[Lemma 4.3]{TMNA2014}}, for some subsequence, 
  $$
  \nabla u_n(x)\rightarrow \nabla u(x)~ \text{and}~ u_n(x)\rightarrow u(x)~\text{a.e. in} ~~ \mathbb{R}^N.
  $$
  Using the limit $I'_{\infty}(u_n)u\rightarrow 0,$ we see that $I'_{\infty}(u)u=0,$ from where it follows that $u \in \mathcal{M}_\infty$. Consequently, 
  $$
  c_{\infty}\leq I_{\infty}(u)=I_{\infty}(u)-\frac1\theta I'_{\infty}(u)u.
  $$ 
  Now, by Fatou's lemma,  
\begin{align*}
c_{\infty}&\leq\int_{\mathbb{R}^N}\Phi(|\nabla u|) dx+\int_{\mathbb{R}^N}\Phi(|u|) dx-\int_{\mathbb{R}^N}F(u) dx\\&-\frac1\theta\int_{\mathbb{R}^N}\phi(|\nabla u|)|\nabla u|^2 dx-\frac1\theta \int_{\mathbb{R}^N}\phi(|u|)|u|^2 dx+\frac1\theta \int_{\mathbb{R}^N} f(u)u dx\\
&\leq (1-\frac{l}{\theta})\int_{\mathbb{R}^N}[\Phi(|\nabla u|)+\Phi(|u|)] dx+\int_{\mathbb{R}^N}\frac1\theta f(u)u-F(u) dx\\
&\leq\liminf_{n \to \infty}\left[(1-\frac{l}{\theta})\int_{\mathbb{R}^N}[\Phi(|\nabla u_n|)+\Phi(|u_n|)] dx+\int_{\mathbb{R}^N}\frac1\theta f(u_n)u_n-F(u_n) dx\right]\\
&\leq c_{\infty}
\end{align*}  
which leads to 
$$
\lim_{n\rightarrow\infty}\int_{\mathbb{R}^N} (\Phi(|\nabla u_n|)+\Phi(|u_n|)) dx=\int_{\mathbb{R}^N}(\Phi(|\nabla u|)+\Phi(|u|)) dx.
$$ 
The above limit combined with $\Delta_2-$condition gives 
$$
u_n\rightarrow u ~\text{in}~ W^{1,\Phi}(\mathbb{R}^N).
$$ 

\noindent \textbf{Case 2: $u=0.$} ~~ In this case, we claim that there are $R,\eta>0$ and $(y_n) \subset \mathbb{R}^N$ which satisfy  
$$
\limsup_{n\rightarrow\infty}\int_{B_{R}(y_n)}\Phi(|u_n|) dx\geq \eta >0.
$$
If this is not true, we must have 
$$
\lim_{n\rightarrow\infty}\sup_{y\in \mathbb{R}^N}\int_{B_R(y)}\Phi(|u_n|) dx=0.
$$
Then, by \protect{\cite[Theoreme 1.3]{TMNA2014}}, 
$$
u_n\rightarrow0~\text{in}~ L^B(\mathbb{R}^N).
$$
The above limit together with $(f_1)$ implies that  $\int_{\mathbb{R}^N} f(u_n)u_ndx\rightarrow0.$
As $u_n\in\mathcal{M}_{\infty}$, we obtain that 
$$\int_{\mathbb{R}^N} \phi(|\nabla u_n|)|\nabla u_n|^2 dx+\int_{\mathbb{R}^N} \phi(|u_n|)|u_n|^2 dx\rightarrow 0.
$$
Then by $(\phi_2)$,  we derive that $I_{\infty}(u_n)\rightarrow0~\text{in}~ W^{1,\Phi}(\mathbb{R}^N)$, which is an absurd, because $I_\infty(u_n) \to c_{\infty}>0$.

Setting $v_n=u_n(x+y_n)$, we derive that $I_{\infty}(v_n)\rightarrow c_{\infty}~\text{and}~ I'_{\infty}(v_{n})\rightarrow0.$  Then, $(v_n)$ is clearly bounded in $W^{1,\Phi}(\mathbb{R}^N)$ and there exists $v\in W^{1,\Phi}(\mathbb{R}^N)$ with $v\neq0$ such that $$v_n\rightharpoonup v~\text{in}~ W^{1,\Phi}(\mathbb{R}^N).$$
Using the same arguments of Case $1$, $v_n\rightarrow v~\text{in}~ W^{1,\Phi}(\mathbb{R}^N).$

Next, we will show that $|y_n|\rightarrow \infty.$ If this does not hold,  $(y_n)$ is bounded in $\mathbb{R}^{N}$ for some subsequence, and  there exists $R'>0$ such that $B_{R}(y_n)\subset B_{R'}(0).$  Hence 
$$
\int_{B_{R'}(0)}\Phi(|u_n|)\,dx \geq \eta>0, \quad \forall n \in \mathbb{N}.
$$
As $u_n\rightharpoonup 0$ in $W^{1,\Phi}(\mathbb{R}^N)$ and $W^{1,\Phi}(\mathbb{R}^N)$ is compactly embedded in $L^{\Phi}(B_{R'}(0))$, we have that $u_n\rightarrow 0$ in $L^{\Phi}(B_{R'}(0))$, which contradicts the last inequality.  

\end{proof}

The next two results are related to the functional $I_\lambda$ and they will be used later on.

\begin{proposition} \label{Pb}
The functional $I_{\lambda}$ satisfies the Palais-Smale condition on $\mathcal{M}_{\lambda}, $ that is , if $(u_n)\subset \mathcal{M}_{\lambda}$ satisfies 
$$
I_{\lambda}(u_n)\rightarrow c~\mbox{and}~ ||I'_{\lambda}(u_n)||_{\star}\rightarrow 0.
$$
then there exists a subsequence, still denoted by $(u_n)$ which is strongly convergent in $W_0^{1,\Phi}(\Omega_{\lambda}).$ Here, $||I'_{\lambda}(v)||_{\star}$ denotes the norm of the derivative of the restriction of $I_\lambda$ to  $\mathcal{M}_\lambda$ at $v$.
\end{proposition}
\begin{proof} Repeating the same arguments explored in the proof of Theorem \ref{310} , we can assume that $(u_n)$ is a $(PS)_c$ sequence for $I_\lambda$, that is, 
$$
I_{\lambda}(u_n)\rightarrow c~\mbox{and}~ ||I'_{\lambda}(u_n)|| \rightarrow 0.
$$
Now, as $\Omega_\lambda$ is bounded, the same type of arguments found in \cite[Section 4]{TMNA2014} guarantee that $I_\lambda$ verifies the $(PS)$ condition, and the proof is complete. 

\end{proof}

The next proposition shows that critical points of $I_{\lambda}$ on $\mathcal{M}_{\lambda}$ are critical point of $I_{\lambda}$ in $W_0^{1,\Phi}(\Omega_{\lambda}).$

\begin{proposition}
If $u\in \mathcal{M}_{\lambda}$ is a critical point of $I_{\lambda}$ on $\mathcal{M}_{\lambda},$ then $u$ is a nontrivial critical point of $I_{\lambda}$ in $W_0^{1,\Phi}(\Omega_{\lambda}).$ Moreover, $u \in C^{1,\alpha}(\overline{\Omega_\lambda})$ and $u(x)>0$ for all $x \in \Omega_\lambda$.  
\end{proposition}
\begin{proof}
Suppose that $u\in \mathcal{M}_{\lambda}$ is a critical point of $I_{\lambda}$ on $\mathcal{M}_{\lambda}.$ Then $u\neq 0$ and there exists $\gamma\in \mathbb{R}$ such that 
$$
I'_{\lambda}(u)=\gamma E'_{\lambda}(u) \quad (\mbox{See Willem} \,\, \cite{W}).
$$ 
As $I'_{\lambda}(u)u=0$, we have that $\gamma E'_{\lambda}(u)u=0.$ From condition $(\phi_3)$ and $(f_4)$, 
$$
-E'_{\lambda}(u)u\geq \int_{\Omega_{\lambda}}(f'(u)-(m-1) f(u))u \, dx>0.
$$ 
Then, $\gamma=0$ and $I'_{\lambda}(u)=0,$ from where it follows that $u$ is a critical point of $I_\lambda$. By \cite{L} and \cite{TF}, we deduce that $u \in C^{1,\alpha}(\overline{\Omega_\lambda})$ for some $\alpha \in (0,1)$. Since we are supposing $f(t)=0$ for $t \leq 0$, we have that $I'(u)u^{-}=0$, where $u^{-}=\min\{u,0\}$. As
$$
I'(u)u^{-}=\int_{\Omega_\lambda}(\phi(|\nabla u^{-}|)|\nabla u^{-}|^{2}+\phi(|u^{-}|)|u^{-}|^{2})\,dx,
$$ 
the condition $(\phi_2)$ yields $u^{-}=0$, then $u(x) \geq 0$ for all $x \in \Omega_\lambda$.  Now, the positiveness of $u$ follows from \cite[Theorem 1.1]{T}( see also \cite{JMP2014} ).
\end{proof}

\section{Behavior of minimax levels}

This section is concerned with the  study of the  behavior of some minimax levels which are crucial in our approach. To do this, we need to fix some notations and definitions. 

In what follows, we assume without loss of generality that $0\in \Omega.$ Furthermore, we fix a real number $r>0$ such that the sets $\Omega_+$ and $\Omega_-$ given by
$$
\Omega_+=\{x\in \mathbb{R}^N, d(x,\overline{\Omega})\leq r\}
$$ 
and 
$$
\Omega_-=\{x\in \Omega, d(x,\partial\Omega)\geq r\}
$$ 
are homotopically equivalent to $\Omega.$ Moreover, for each $x\in \mathbb{R}^N$ and $R>r>0,$ we define 
$$
A_{R,r,x} =B_R(x)\setminus \overline{B}_r(x). 
$$ 
Hereafter, we denote by $A_{R,r}$ the set $A_{R,r,0}.$ 

For each $u\in W^{1,\Phi}(\mathbb{R}^N)$ with compact support, we consider 
$$
\beta(u)=\frac{\int_{\mathbb{R}^N}x\Phi(|\nabla u|)dx}{\int_{\mathbb{R}^N}\Phi(|\nabla u|)dx} \,\,\,\,\,\,\, ( \mbox{Barycenter function ) }
$$
and for each $x\in \mathbb{R}^N$, we set $a (R,r,\lambda,x)$ by 
$$
a (R,r,\lambda,x)=\inf\{J_{\lambda,x}(u), \beta(u)=x, u\in \widehat{\mathcal{M}}_{\lambda,x}\},
$$
where 
$$
J_{\lambda,x}(u)=\int_{A_{\lambda R,\lambda r,x}}(\Phi(|\nabla u|)+\Phi(|u|)) dx-\int_{A_{\lambda R,\lambda r,x}} F(u) dx
$$ 
and  
$$
\widehat{\mathcal{M}}_{\lambda,x}=\{u\in W_0^{1,\Phi}(A_{\lambda R,\lambda r, x})\setminus\{0\}: J'_{\lambda,x}(u)u=0\}.
$$
In the sequel, $a(R,r,\lambda), J_{\lambda}~\text{and}~ \widehat{\mathcal{M}}_{\lambda}$ denote $a(R,r,\lambda,0), J_{\lambda,0}$ and $\widehat{\mathcal{M}}_{\lambda,0}$ respectively.
\begin{proposition}\label{4.1} The number $a(R,r,\lambda)$ satisfies $$\liminf_{\lambda\rightarrow\infty} a(R,r,\lambda)>c_{\infty}.$$
\end{proposition} 
\begin{proof}
From definitions of $a(R,r,\lambda)$ and $c_{\infty}$, we know that 
$$
a(R,r,\lambda)\geq c_{\infty}.
$$
Therefore, $$ \liminf_{\lambda\rightarrow+\infty} a(R,r,\lambda)\geq c_{\infty}.$$
Suppose by contradiction that 
$$ 
\liminf_{\lambda\rightarrow+\infty} a(R,r,\lambda)= c_{\infty}.
$$
Then there exists $\lambda_n\rightarrow\infty$ and $u_n\in \widehat{\mathcal{M}}_{\lambda_n}$ such that 
$$
\beta(u_n)=0 ~\text{and}~ a(R,r,\lambda_n)\rightarrow c_{\infty}.
$$ 
By Theorem \ref{310},   
$$
u_n(x)=w_n(x)+v(x-y_n),
$$ 
where  $(w_n)\subset W^{1,\Phi}(\mathbb{R}^N)$ converges strongly to $0$ in $W^{1,\Phi}(\mathbb{R}^N),$ $(y_n)\subset \mathbb{R}^N$ satisfies $|y_n|\rightarrow\infty$ and $v\in W^{1,\Phi}(\mathbb{R}^N)$ is a positive function with 
$$
I_{\infty}(v)=c_{\infty}~\text{and}~ I'_{\infty}(v)=0.
$$
As $I_{\infty}$ is rotationally invariant, we can assume that 
$$
y_n=(y_n^1,0,0,\ldots,0) \quad \mbox{and} \quad y_n^1<0.
$$
Setting 
$$
M=\int_{\mathbb{R}^N} \Phi(|\nabla v|) dx>0,
$$ 
a direct computation gives 
$$
\int_{B_{r\lambda_n/2}(y_n)}\Phi(|\nabla( w_n+v(.-y_n))|) dx \rightarrow M.
$$ 
In the sequel, we consider the two following sets:
$$
\Theta_n= B_{r\lambda_n/2}(y_n)\cap [B_{\lambda_n R}(0)\setminus \overline{B}_{\lambda_n r}(0)]
$$ 
and 
$$
\Gamma_n=\left[B_{\lambda_n R}(0)\setminus  \overline{B}_{\lambda_n r}(0) \right]\setminus B_{r\lambda_n/2 }(y_n).
$$
As $(u_n)\subset W_0^{1,\Phi}(A_{\lambda_n R,\lambda_n r}),$  
\begin{equation}
\int_{B_{r\lambda_n/2}(y_n)} \Phi(|\nabla u_n|) dx= \int_{A_{\lambda_n R,\lambda_n r}\cap B_{{r\lambda_n/2}}(y_n)} \Phi(|\nabla u_n|) dx=\int_{\Theta_n}\Phi(|\nabla u_n|) dx.
\end{equation}
From this,  
\begin{equation}\label{M}
\int_{\Theta_n}\Phi(|\nabla u_n|) dx \rightarrow M
\end{equation}
and 
\begin{equation}\label{41}
\int_{\Gamma_n}\Phi(|\nabla  u_n|) dx \rightarrow 0. 
\end{equation}
Since $\beta(u_n)=0$, we know that 
\begin{equation}\label{413}
0=\int_{A_{\lambda_n R, \lambda_n r}}x_1 \Phi(|\nabla u_n|) dx=\int_{\Theta_n} x_1 \Phi(|\nabla u_n|) dx+\int_{\Gamma_n} x_1 \Phi(|\nabla u_n|) dx.
\end{equation}
From the definition of $\Gamma_n$,  
\begin{equation}\label{414}
\int_{\Gamma_n} x_1 \Phi(|\nabla u_n|) dx\leq R\lambda_n \int_{\Gamma_n}\Phi(|\nabla u_n|) dx.
\end{equation}
On another side, if $x\in \Theta_n$, then $x\in B_{r\lambda_n/2}(y_n)$ and $x\notin \overline{B}_{r\lambda_n}(0)$. Hence,  
$$|x_1-y_n^1|^2+\sum_{i=2}^N|x_i|^2\leq \frac{r^2\lambda_n^2}{4} ~~\text{and}~~ \sum_{i=2}^N|x_i|^2>r^2\lambda_n^2-|x_1|^2,
$$
from where it follows that 
$$
|x_1|\geq \frac{\sqrt{3}r\lambda_n}{2}>\frac{r\lambda_n}{2}.
$$
The above inequality together with 
$$
|x_1-y_n^1|^2+\sum_{i=2}^N|x_i|^2\leq \frac{r^2\lambda_n^2}{4}
$$ 
implies that $x_1<-\frac{r\lambda_n}{2}.$ This combine with $(\ref{M})$ to give 
\begin{equation}\label{415}
\int_{\Theta_n} x_1\Phi(|\nabla u_n|) dx\leq -\frac{r\lambda_n}{2}(M+o_n(1)). 
\end{equation}
Thereby, $(\ref{414})$, $(\ref{415})$ and $(\ref{413})$ lead to  
\begin{equation}
0=\int_{A_{\lambda_n R,\lambda_n r}} x_1\Phi(|\nabla u_n|) dx\leq -\frac{r\lambda_n}{2}(M+o(1))+R\lambda_n \int_{\Gamma_n} \Phi(|\nabla u_n|) dx,
\end{equation}
or equivalently,  
$$
-(r\lambda_n/2)(M+o_n(1))+R\lambda_n\int_{\Gamma_n} \Phi(|\nabla u_n|) dx\geq0.
$$ 
Thus 
$$
\int_{\Gamma_n}\Phi(|\nabla u_n|) dx \geq \frac{rM}{2R}-o_n(1),
$$ 
which contradicts $(\ref{41})$.
\end{proof}

In what follows, let us denote  by $b_\lambda$ the mountain pass level of the energy functional $I_{\lambda,B}: W_0^{1,\Phi}(B_{\lambda r}) \to \mathbb{R}$ given by
$$
I_{\lambda,B}(u)=\int_{B_{\lambda r}} \Phi(|\nabla u|)\, dx+\int_{B_{\lambda r}} \Phi(|u|)\, dx-\int_{B_{\lambda r}} F(u)\, dx,
$$
where $B_{\lambda r}=\lambda B_r(0)$ and by $\mathcal{M}_{\lambda,B}$ the Nehari manifold related to the $I_{\lambda,B}$ given by 
$$
\mathcal{M}_{\lambda,B}=\{u \in W_0^{1,\Phi}(B_{\lambda r}) \setminus \{0\}\,:\,I'_{\lambda,B}(u)u=0\}.
$$
Repeating the same arguments explored in \cite{JMP2014}, it is possible to prove that 
$$
b_\lambda=\inf_{u \in \mathcal{M}_{\lambda,B}}I_{\lambda,B}(u).
$$

The next result will be used to study the behavior of barycenter of some sequences, which will appear in the proof of Proposition \ref{P3.2} below.

\begin{proposition}\label{4.2} The numbers $b_{\lambda}$ and $c_{\lambda}$ verify the following limits $$\lim_{\lambda\rightarrow\infty} c_{\lambda}=c_{\infty}~\mbox{and}~ \lim_{\lambda\rightarrow\infty} b_{\lambda}=c_{\infty}.$$
\end{proposition}
\begin{proof}
We will prove only the first limit, because the second one follows with the same arguments. 
Let $h \in C^{\infty}_0(\mathbb{R}^N,[0,1])$ with  
$$
h(x)=\left\{
\begin{array}{l}
1,~~\text{in}~ B_1(0),\\ 
0,~~\text{in}~ B^c_2(0). \\
\end{array}
\right.
$$
For each $R>0$, let us consider the function $h_{R}(x)=h(x/R)$ and $w_{R}(x)=h_{R}(x) w(x),$ where $w$ is a ground state solution of $(P_{\infty}).$ Since $0\in \Omega,$ there exists $\lambda^{\star}>0$ such that $B_{2R}(0)\subset \Omega_{\lambda}$ for $\lambda\geq \lambda^{\star}.$ Let $t_{R}>0$ satisfy $t_R w_R\in \mathcal{M}_{\lambda}$. Then 
$$
c_{\lambda}\leq I_{\lambda}(t_{R}w_{R}),~ \forall \lambda\geq \lambda^{\star}.
$$ 
Taking the limit when $\lambda\rightarrow\infty,$ we obtain 
$$
\limsup_{\lambda\rightarrow\infty} c_{\lambda}\leq I_{\infty}(t_{R}w_{R}).
$$

\noindent \textbf{Claim 1:} $\displaystyle \lim_{R\rightarrow\infty}t_R=1.$

By definition of $t_R$, 
$$
t_Rw_R\in \mathcal{M}_{\lambda}\Longleftrightarrow I'_{\infty}(t_Rw_R)(t_Rw_R)=0,
$$ 
or equivalently,
\begin{equation} \label{2}
\int_{\mathbb{R}^N}\phi(t_R|\nabla w_R|)(t_R|\nabla w_R|)^2dx+\int_{\mathbb{R}^N}\phi(t_R|w_R|)(t_R|w_R|)^2 dx=\int_{\mathbb{R}^N}f(t_Rw_R)(t_Rw_R) dx.
\end{equation}
So, for $R>1$, 
$$
\int_{\mathbb{R}^N}\phi(t_R|\nabla w_R|)(t_R|\nabla w_R|)^2 dx+\int_{\mathbb{R}^N}\phi(t_R|w_R|)(t_R|w_R|)^2dx\geq \int_{B_1(0)}f(t_Rw)(t_Rw)dx\geq\int_{B_1(0)} f(t_R a)(t_R a)dx
$$
where $a=\displaystyle\min_{|x|\leq1} w(x).$

Now, gathering $(\phi_2)$ and \protect{\cite[lemma 2.3]{TMNA2014}}, we derive that 
\begin{align*}
\int_{\mathbb{R}^N} \phi(t_R|\nabla w_R|)(t_R|\nabla w_R|)^2 dx+\int_{\mathbb{R}^N} \phi(t_R| w_R|)(t_R|w_R|)^2 dx&\leq\int_{\mathbb{R}^N} m \Phi(t_R |\nabla w_R|) dx+\int_{\mathbb{R}^N} m\Phi(t_R|w_R|) dx\\
&=m\int_{\mathbb{R}^N} \Phi(t_R|\nabla w_R|)+\Phi(t_R|w_R|) dx\\
&\leq m\xi_1(t_R) \int_{\mathbb{R}^N}\Phi(|\nabla w_R|)+\Phi(|w_R|) dx,
\end{align*} 
and so
$$
\int_{\mathbb{R}^N}\Phi(|\nabla w_R|)+\Phi(|w_R|) dx\geq \frac{1}{m\xi_1(t_R)} \int_{B_1(0)} f(t_R a) t_R a\,dx,
$$ 
where $\xi_1(t_R)=\max\{t_R^l,t_R^m\}.$ Using the above information, we are able to prove that $(t_R)$ is bounded. In fact, if there exists $R_n\rightarrow \infty$ with $t_{R_n}\rightarrow\infty, $ we ensure that $\xi_1(t_{R_n}) =t_{R_n}^m$ (because $m>l$), then 
\begin{equation}\label{1}
\int_{\mathbb{R}^N} \Phi(|\nabla w_{R_n}|)+\Phi(|w_{R_n}|) dx\geq \frac{1}{m t_{R_n}^m}\int_{B_1(0)} f(t_{R_n} a) t_{R_n} a \,dx.
\end{equation} 
Thereby, by $(f_2)$,  
$$
\int_{\mathbb{R}^N} (\Phi(|\nabla w_{R_n}|)+\Phi(|w_{R_n}|) )dx\geq \frac{\theta F(t_0)}{mt_0^{\theta}}\int_{B_1(0)} t^{\theta-m}_{{R_n}} a^{\theta} dx.
$$ 
As $\theta>m,$ 
$$
t_{R_n}^{\theta-m}\overset{n\rightarrow\infty}{\longrightarrow}+\infty,
$$
which yields 
$$
\int_{\mathbb{R}^N}\Phi(|\nabla w_{R_n}|)+\Phi(|w_{R_n}|) dx\rightarrow+\infty,
$$   
that is,
$$
||w_{R_{n}}||\rightarrow\infty,
$$ 
which is an absurd, because $||w_{R_n}||\rightarrow ||w||~\text{in}~ W^{1,\Phi}(\mathbb{R}^N)$. Then $(t_{R})$ is bounded. Now, we will show that there is no $R_n\rightarrow +\infty$ such that $\overset{n\rightarrow \infty}{ t_{R_n}\longrightarrow 0}.$  Indeed, from \protect{\cite[lemma 4.1]{JMP2014}}, as $t_{{R_n}}w_{{R_n}}\in \mathcal{M}_{\lambda}$, there exists $\alpha >0$ such that 
$$
||t_{R_n} w_{R_n}||\geq \alpha, ~ \forall n\in \mathbb{N}
$$ 
and then 
$$
t_{R_n}>\frac{\alpha}{||w_{R_n}\|}.
$$ 
Since $||w_{R_n}||\rightarrow||w||$, we conclude that 
$$
\liminf_{n\rightarrow+\infty} t_{R_n} >0.
$$
Therefore, there exist $R_0,\delta>0$ such that $t_{R}>\delta$ for $R\geq R_0.$
Fixing $R_n \rightarrow +\infty$ with  $t_{R_n}\rightarrow t_0$, it follows from $(\ref{2})$  
\begin{equation}\label{A}
\int_{\mathbb{R}^N}\phi(t_0|\nabla w|)(t_0|\nabla w|)^2dx+\int_{\mathbb{R}^N}\phi(t_0|w|)(t_0|w|)^2 dx=\int_{\mathbb{R}^N}f(t_0 w)(t_0 w) dx.
\end{equation}
By $(\phi_3)$ and $(f_3)$, it is easy to check that $t_0=1.$ Consequently, 
$$
I_{\infty}(t_R w_R)\rightarrow I_{\infty}(w)=c_{\infty}~\text{when}~ R\rightarrow +\infty,
$$ 
and
\begin{equation}\label{44}
\limsup_{\lambda\rightarrow\infty} c_{\lambda}\leq c_{\infty}.
\end{equation}
On the other hand, from the definition of $c_{\lambda}$ and $c_{\infty}$, we get the inequality  
$$
c_{\lambda}\geq c_{\infty}, \quad \forall \lambda>0,
$$
which leads to 
\begin{equation}\label{45}
\liminf_{\lambda\rightarrow +\infty} c_{\lambda}\geq c_{\infty}.
\end{equation}
From $(\ref{44})$ and $(\ref{45})$, 
$$
\lim_{\lambda\rightarrow +\infty} c_{\lambda}=c_{\infty}.
$$
\end{proof}

The proposition below is crucial to apply the Lusternik - Schnirelman Theory.
\begin{proposition} \label{P3.2}
There exists $\widehat{\lambda}>0$ such that : $$I_{\lambda}(u)\leq b_{\lambda}~\text{and}~ u\in \mathcal{M}_{\lambda}\Rightarrow \beta(u)\in \lambda\Omega^+_r,~\forall \lambda\geq \widehat{\lambda}.$$
\end{proposition}
\begin{proof} 
Assume by contradiction that the lemma does not occur. Then, there exist $\lambda_n\rightarrow +\infty, u_n\in\mathcal{M}_{\lambda_n}$ and $I_{\lambda_n}(u_n)\leq b_{\lambda_n}$  such that 
$$
x_n=\beta(u_n)\notin \lambda_n\Omega^+_r.
$$
Fixing $R>diam \Omega,$ we have  
\begin{equation}\label{424}
\Omega_{\lambda_n}\subset A_{\lambda_n R,\lambda_n r,x_n}.
\end{equation}
In fact, for $y\in \Omega_{\lambda_n}$, 
\begin{align*}
|y-x_n|&=\left|\frac{\int_{\Omega_{\lambda_n}}y\Phi(|\nabla u_n|) dz}{\int_{\Omega_{\lambda_n}}\Phi(|\nabla u_n|) dz}-\frac{\int_{\Omega_{\lambda_n}}z\Phi(|\nabla u_n|) dz}{\int_{\Omega_{\lambda_n}}\Phi(|\nabla u_n|) dz}\right|\\&=\left|\frac{\lambda_n\int_{\Omega_{\lambda_n}}(x-\frac{z}{\lambda_n})\Phi(|\nabla u_n|) dz}{\int_{\Omega_{\lambda_n}}\Phi(|\nabla u_n|) dz}\right|\leq \lambda_n diam\Omega\leq \lambda_n R.
\end{align*} 
Then,  
\begin{equation}\label{428}
|x_n-y|\leq R\lambda_n.
\end{equation}
which shows $(\ref{424})$.

By using of the definition of $a(R,r,\lambda_n,x_n)$ and the fact that $a(R,r,\lambda_n,x_n)= a(R,r,\lambda_n),$  we get 
$$ 
a(R,r,\lambda_n)\leq b_{\lambda_n}.
$$ 
Then, by Proposition \ref{4.2},  
$$
\liminf_{n\rightarrow\infty}a(R,r,\lambda_n)\leq c_{\infty},
$$  
which contradicts the Proposition \ref{4.1}.
\end{proof}
\begin{proposition} The functional $I_{\lambda,B}$ has a ground state solution $u_{\lambda,r}$ which is radially symmetric on the origin.
\end{proposition}
\begin{proof}
Let $v \in W_0^{1,\Phi}(B_{\lambda r})$ be a positive ground state solution for $I_{\lambda,B}$ that is 
$$
I_{\lambda,B}(v)=b_{\lambda}~\text{and}~ I'_{\lambda,B}(v)=0.
$$
If $v^{\star}$ is the Schwartz symmetrization of $v$, the Pólya–Szegö principle ensures that $v^{\star}\in W_0^{1,\Phi}(B_{\lambda r}(0))$ and  
\begin{equation}\label{47}
\int_{B_{\lambda r}(0)} \Phi(|\nabla v^{\star}|) dx\leq \int_{B_{\lambda r}(0)}\Phi(|\nabla v|) dx.
\end{equation}
On the other hand, we also have 
\begin{equation}\label{48}
\int_{B_{\lambda r}(0)} F(\alpha v^*) dx=\int_{B_{\lambda r}(0)} F(\alpha v) dx,~\forall \alpha>0.
\end{equation}
From \protect{\cite[lemma 3.1]{JMP2014}}, there exists a unique $t^*>0$ such that $t^*v^*\in \mathcal{M}_{\lambda,B}.$ Thereby, from $(\ref{47})$ and $(\ref{48})$,  
$$
b_{\lambda}\leq I_{\lambda,B}(t^*v^*)\leq I_{\lambda,B}(t^* v)\leq \max_{t\geq0} I_{\lambda,B}(tv)=I_{\lambda,B}(v)=b_{\lambda},
$$
and so,  
$$
I_{\lambda,B}(t^*v^*)=b_{\lambda}~~\text{and}~~t^*v^*\in\mathcal{M}_{\lambda,B}.
$$
Consequently, $t^*v^*$ is a critical point of $I_{\lambda,B}$ on $\mathcal{M}_{\lambda,B}$, then $u_{\lambda,r}=t^*v^* \in W_0^{1,\Phi}(B_{\lambda,r})$ is radially symmetric on the origin and satisfies  
$$
I_{\lambda,B}(u_{\lambda,r})=b_{\lambda}~\text{and}~ I'_{\lambda,B}(u_{\lambda,r})=0.
$$ 
 
\end{proof}

In the sequel, for each $\lambda>0$ and $r>0$, we define the operator $\Psi_r: \lambda\Omega_-\rightarrow W_0^{1,\Phi}(\Omega_{\lambda})$ by 
$$
[\Psi_r(y)](x)=\begin{cases} u_{\lambda,r}(|x-y|), ~~\forall x\in B_{\lambda r}(y),\\ 0,~~~~~~~~~~~~~~~~~~~\forall x\in \Omega_{\lambda}\setminus B_{\lambda r}(y).\end{cases}
$$

\begin{proposition} \label{cat} For $\lambda\geq \widehat{\lambda},$ we have 
$$ 
cat (I^{b_{\lambda}}_{\lambda})\geq cat(\Omega),
$$
where $I^{b_{\lambda}}_{\lambda}=\{u\in \mathcal{M}_{\lambda}; I_{\lambda}(u) \leq b_{\lambda}\}$.  

\end{proposition}
\begin{proof}
If we assume that 
$$
I^{b_{\lambda}}_{\lambda}=A_1\cup \ldots \cup A_n,
$$ 
where $A_i, i=1\ldots,n$ are closed and contractible in $I^{b_{\lambda}}_{\lambda}$, then there exists a continuous function $h_j: [0,1]\times A_j\rightarrow I^{b_{\lambda}}_{\lambda}$ such that 
$$
h_j(0,u)=u ~\text{and}~ h_j(1,u)=z_j~\text{for all}~ u\in A_j,
$$ 
where $z_j$ is a fixed element in $A_j.$ Consider $B_j=\Psi^{-1}_r(A_j), 1\leq j\leq n.$ Then $B_j$ are closed and 
$$
\lambda\Omega_-=B_1\cup\ldots\cup B_n.
$$ 
Setting the deformation $g_j:[0,1]\times B_j\rightarrow \lambda \Omega_+$ given by 
$$
g_j(t,y)=\beta(h_j(t,\Psi_r(y))),
$$ 
we conclude that $B_j$ is contractible in $\lambda\Omega_+$, from where it follows that 
$$
cat(\Omega)=cat(\Omega_{\lambda})=cat_{\lambda\Omega_+}(\lambda\Omega_-)\leq n \leq cat (I^{b_{\lambda}}_{\lambda}).
$$

\vspace{0.5 cm}

\noindent {\bf Proof of Theorem \ref{T1} }  First of all, let us recall that $I_{\lambda}$ satisfies the Palais-Smale condition on $\mathcal{M}_\lambda$. Thus, by applying of Lusternik - Schnirelman Theory and Proposition 
\ref{cat}, we assure that $I_\lambda$ on $\mathcal{M}_\lambda$ has at least $cat(\Omega)$ critical points whose energy is less than $b_\lambda$ for $\lambda \geq \widehat{\lambda}$. 
 
\end{proof}

\end{document}